\documentclass[12pt,reqno]{amsart}
\usepackage{etoolbox}
\makeatletter
\patchcmd\maketitle
{\uppercasenonmath\shorttitle}
{}
{}{}
\patchcmd\maketitle
{\@nx\MakeUppercase{\the\toks@}}
{\the\toks@}
{}
{}{}
\patchcmd\@settitle{\uppercasenonmath\@title}{\Large}{}{}
\patchcmd\@setauthors
{\MakeUppercase{\authors}}
{\authors}
{}{}
\makeatother
\usepackage{amsmath, amsthm, amscd, amsfonts, amssymb, graphicx, color}
\usepackage{url}
\usepackage{tikz-cd}
\usepackage{hyphenat}
\usepackage[utf8]{inputenc}
\usepackage[T1]{fontenc}
\newtheorem{theorem}{Theorem}[section]

\newtheorem{corollary}{Corollary}[section]
\newtheorem{proposition}{Proposition}[section]
\newtheorem{lemma}{Lemma}[section]
\newtheorem{remark}{Remark}[section]
\newtheorem{example}{Example}[section]

\numberwithin{equation}{section}

\usepackage[colorlinks=true]{hyperref}
\hypersetup{urlcolor=blue, citecolor=red , linkcolor= blue}
\usepackage[capitalise,noabbrev,nameinlink]{cleveref}      

\def\bh{\mathcal{B}(\mathcal{H})}
\def\cH{\mathcal{H}}

\def\h{\mathcal{H}}
\def\R{\mathbb{R}}
\def\C{\mathbb{C}}

\begin{document}
\author[T. Bottazzi and  C. Conde ] {\Large{Tamara Bottazzi}$^{1_{a,b}}$ and  \Large{Cristian Conde}$^{1_b, 2}$}

\address{$^{[1_a]}$ Universidad Nacional de R\'io Negro. Centro Interdisciplinario de Telecomunicaciones, Electrónica, Computación Y Ciencia Aplicada (CITECCA), Sede Andina (8400) S.C. de Bariloche, Argentina.}
\address{$^{[1_b]}$ Consejo Nacional de Investigaciones Cient\'ificas y T\'ecnicas, (1425) Buenos Aires,
	Argentina.}
\email{\url{tbottazzi@unrn.edu.ar}}

\address{$^{[2]}$ Instituto de Ciencias, Universidad Nacional de Gral. Sarmiento, J. M. Gutierrez 1150, (B1613GSX) Los Polvorines, Argentina}
\email{\url{cconde@campus.ungs.edu.ar}}

\keywords{Buzano inequality, Cauchy-Schwarz inequality, Inner product space, Hilbert space, Bounded linear operator.}
\subjclass[2020]{Primary: 46C05, 26D15. Secondary: 47B65, 47A12.}
\date{\today}

\title[ Generalized Buzano Inequality]
{ Generalized Buzano Inequality }
\maketitle

\begin{abstract}
If $P$ is an orthogonal projection defined on an inner product space $\mathcal{H}$, then the inequality
$$
|\langle Px, y\rangle|\leq \frac12 [\|x\|\|y\|+|\langle x, y\rangle|]
$$
fulfills  for any $x,y \in \mathcal{H}$ (see \cite{Dra16}). In particular, when $P$ is the identity operator, then it recovers the famous Buzano inequality.
We obtain generalizations of such classical inequality, which hold for certain families of bounded linear operators defined on $\mathcal{H}$. In addition, several new inequalities involving the norm and numerical radius of an operator are established.
\end{abstract}

\section{Introduction}\label{s1}

Let $(\mathcal{H}, \langle \cdot, \cdot\rangle)$
be an inner product space over the real or complex numbers field $\mathbb{K}$. The
following inequality is well known in literature as the Cauchy-Schwarz  inequality
\begin{equation}\label{CS}
	|\langle x, y\rangle|\leq \|x\| \|y\|,
\end{equation}
for any $x, y \in \mathcal{H}$.   The equality in  \eqref{CS} holds if and only if there exists a constant $\alpha \in \mathbb{K}$ 
such that $x = \alpha y.$

In \cite{buzano}, Maria Luisa Buzano gave the following extension of the celebrated Cauchy– Schwarz  inequality in $\mathcal{H}$
\begin{equation}\label{buzano}
	|\langle x, z\rangle \langle z, y\rangle|\leq \frac12(|\langle x, y\rangle|+\|x\| \|y\| )\|z\|^2,
\end{equation}
for any $x,y, z\in \mathcal{H}.$ Last inequality  is called Buzano  inequality.

The original proof of Buzano has it difficulty  since it requires some facts about orthogonal decomposition of a complete inner product space. 

In \cite{Dra85}, Dragomir established a  refinement
of \eqref{CS} which implies the Buzano inequality.
Moreover, Fuji and Kubo \cite{fujiikubo} gave a simpler proof of \eqref{buzano} by using an orthogonal projection on a subspace of  $\mathcal{H}$ and \eqref{CS}. Furthermore, they characterized when the equality holds.

This paper aims to present new generalizations of Buzano inequality and it is organized as follows. Section 2 contains some definitions and usual results about bounded linear operators defined on a Hilbert space. In Section 3, we present and prove the $\frac{1}{\alpha}$-Buzano inequality (if $\alpha=2$ gives the classical Buzano inequality) and it is devoted to describing different families of operators which fulfill such inequality for different values of the parameter $\alpha$.  Finally, in Section 4 relates the distinct inequalities previously obtained with the numerical radius, improving new bounds for the last one.

\section{Preliminaries}\label{s2}

As any pre-Hilbert space can be completed to a Hilbert space, from now on, we suppose that $\mathcal{H}$ is a Hilbert space. 
Let $\mathcal{B}(\mathcal{H})$ denote the $C^*$-algebra of all bounded linear operators acting on a separable non trivial complex Hilbert space
$\mathcal{H}$ with an inner product $\langle\cdot,\cdot\rangle$ and the corresponding norm $\|\cdot\|$. The symbol $I$ stands for the identity operator and 
$\mathcal{GL}(\mathcal{H})$ denotes the group of invertible operators on $\mathcal{H}$.

The range of every operator is denoted by $\mathcal{R}(T)$, its null space by $\mathcal{N}(T)$. If $T\in \mathcal{B}(\mathcal{H}),$ we say that $T$ is a positive operator,  $T\geq 0$,   whenever $\langle Tx,x\rangle \geq 0$ for all $x\in \mathcal{H}$ and we denote by  $\bh^+$, the subset of all positive bounded linear operators definded on $\cH$. The definition of
positivity induces the order $T\geq S$ for self-adjoint operators if and only if  $T-S\geq 0.$   For any $T\in \bh^+$, there exists a unique positive $T^{1/2}\in \mathcal{B}(\mathcal{H})$ such that $T=(T^{1/2})^2$.
Let  $T^*$ be the adjoint of $T$  and $|T|=(T^*T)^{1/2}.$

The polar decomposition theorem asserts that for every operator $T\in \bh$  there is a partial isometry $V\in \bh$ such that can be written as the product $T=V|T|$. In particular, $V$ satisfying
$\mathcal{N}(V)=\mathcal{N}(T)$
exists and is uniquely determined.

For any $T\in \bh$, we denote by $\sigma(T)$ its spectrum and by $\sigma_{app}(T)$ its approximate point spectrum, that is
$$\sigma_{app}(T)=\{\lambda\in \mathbb{C}: \exists \:\{x_n\}_{n\in \mathbb N}, \|x_n\|=1 \:{\textrm{and}}\:\lim _{{n\to \infty }}\|Tx_{n}-\lambda x_{n}\|=0\}.$$

For any $T\in \bh$, we define $m(T)=\inf\{\|Tx\|:\ x\in \mathcal{H},\|x\|=1\}$. Clearly, $m(T)\geq 0$  and $m(T)>0$ if and only if $0\notin \sigma_{app}(T)$ (\cite{PHD}).

For a linear operator $T$ on a Hilbert space $\mathcal{H}$, the numerical range $W(T)$ is the
image of the unit sphere of $\mathcal{H}$ under the quadratic form $x\to \langle Tx, x\rangle$. More precisely,
\begin{equation*}
	W(T)=\{\langle Tx, x\rangle : x\in \mathcal{H}, \|x\|=1\}.
\end{equation*}
The numerical range of an operator is a convex subset of the complex plane (\cite{Hal}). Then, for any $T$ in $\mathcal{B}(\mathcal{H})$ we define the numerical radius of $T$, 
\begin{equation*}
	\omega(T)=\sup\{|\lambda|: \lambda \in W(T)\}.
\end{equation*}

It is well-known that $\omega(\cdot)$ defines a norm on   $\mathcal{B}(\mathcal{H})$, and we have for all $T \in \mathcal{B}(\mathcal{H})$, 
\begin{equation}\label{omegaequiv}
	\frac 12 \|T\|\leq \omega(T)\leq \|T\|. 
\end{equation}
Thus, the usual operator norm and the numerical radius are equivalent. Inequalities in \eqref{omegaequiv} are sharp if $T^2=0$, then the first inequality becomes  equality, while the second inequality becomes an equality if $T$ is normal.

For any compact operator $T\in \mathcal{B}(\mathcal{H})$ and $j\in \mathbb{N}$, let $s_j(T)=\lambda_j(|T|)$, be the $j$-th singular value of $T$, i.e.
the $j$-th eigenvalue of $|T|$ in decreasing order and repeated
according to multiplicity.  Let $\rm tr(\cdot)$ be the trace functional,
\begin{equation*}
	{\rm tr}(T)=\sum_{j=1}^{\infty} \langle Te_j,e_j\rangle,
\end{equation*}
where $\{e_j\}_{j=1}^{\infty}$ is an orthonormal basis of $\mathcal{H}$. 
Note that this coincides with the usual definition of the trace if $\mathcal{H}$ is finite-dimensional. 



Let $T=x\otimes y$ be a rank one operator defined by $T(z)=\langle z, y\rangle x$  with $x, y, z \in \mathcal{H}$. Then, by Lemma 2.1 in \cite{ChGLT} and using the well-known fact that ${\rm tr}(x\otimes y)=\langle x, y\rangle$,  we obtain
\begin{equation*}
	\omega(x\otimes y)=\frac 12 \left(|{\rm tr}(x\otimes y)|+\|x\otimes y\|\right)=\frac 12 \left(|\langle x, y\rangle |+\|x\| \|y\|\right).
\end{equation*}
We remark that the numerical radius of the rank one operator $T=x\otimes y$ coincides with the upper bound of Buzano  inequality. From this fact, we are able to give a new proof of inequality \eqref{buzano} using this fact. If $\|z\|=1,$	then $\langle Tz, z\rangle=\langle z, y\rangle \langle x, z\rangle \in W(T)$ and
\begin{equation*}
	|\langle x, z\rangle \langle z, y\rangle|=|\langle Tz, z\rangle|\leq\omega(T) = \frac{1}{2}\left(|\langle x, y\rangle| + \| x\| \| y\|\right).
\end{equation*}

For $T\in \bh$, we have, by definition, 
$$
dist(I,\mathbb{C}T):=\inf_{\gamma \in \mathbb{C}}\|\gamma T-I\|\: \text{ and }  \:	dist(T,\mathbb{C}I):=\inf_{\beta \in \mathbb{C}}\|T-\beta I\|. $$

Evidently there is at least
one complex number $\gamma_0\in \mathbb{C}$ such that $dist(I,\mathbb{C}T)=\|\gamma_0 T-I\|$ and  in addition, if $m(T)>0$ then the value $\gamma_0$ is unique. Following Stampfli \cite{stampfli}, we call such scalar as the center of mass of $T$ and we denote by $c(T).$
For $A, T\in \bh$ such that $m(T)>0$ we consider
\begin{equation}\label{paul}
	M_T(A)=\sup_{\|x\|=1}\left[\|Ax\|^2-\frac{|\langle Ax, Tx\rangle|^2}{\|Tx\|^2}\right]^{1/2}.
\end{equation}
In \cite{Pa}, Paul proved that $M_T(A)=dist(A,\mathbb CT)$.

Given $T, S\in \bh$ we said that $T$ is Birkhoff-James orthogonal to $S$ if and only if $\|T\|\leq \|T-\lambda S\|$ for every $\lambda\in \C$.

\section{$\frac{1}{\alpha}$-Buzano inequality}\label{s3}

In the last decades, several mathematicians presented different proofs of Buzano  inequality. We start by presenting a new and simple proof of such inequality using a rank one operator.

Given $z\in \mathcal{H}$ with $\|z\|=1$ and $\alpha \in \mathbb{C}$, we consider the rank one operator $T=z\otimes z$. Then, for any $u\in \mathcal{H}$, it holds
\begin{equation*}
	\|(\alpha T-I)u\|^2=\|\alpha Tu-u\|^2=(|\alpha-1|^2-1)|\langle z, u\rangle|^2+\|u\|^2\leq \max\{1, |\alpha-1|^2\}\|u\|^2.
\end{equation*}
Hence $\|\alpha T-I\|\leq \max\{1, |\alpha-1|\}$ and for any $x, y \in \mathcal{H}$ we get 
\begin{equation*}\label{newMox}
	|\langle  \left(\alpha T-I\right)x, y\rangle|\leq \|T-I\|\|x\|\|y\|\leq \max\{1, |\alpha-1|\}\|x\| \|y\|.
\end{equation*}
In conclusion, we have 
\begin{equation}\label{alphabuz}
	|\alpha \langle x,z\rangle \langle z, y\rangle -\langle x, y\rangle|\leq \max\{1, |\alpha-1|\}\|x\|\|y\|,
\end{equation}
for any $x,y,z\in \mathcal{H}$ with $\|z\|=1$ and $\alpha \in \mathbb{C}.$
If $\alpha \in \mathbb{C}-\{0\}$,  then \eqref{alphabuz} is equivalent to 
\begin{equation*}
	\left| \langle x,z\rangle \langle z, y\rangle -\frac{1}{\alpha}\langle x, y\rangle\right|\leq \frac{1}{|\alpha|}\max\{1, |\alpha-1|\}\|x\|\|y\|.
\end{equation*}
From  the continuity property of modulus for complex numbers,  we  obtain 
\begin{equation*}
	|\langle x, z\rangle \langle z, y\rangle|\leq \frac{1}{|\alpha|}(|\langle x, y\rangle|+ \max\{1, |\alpha-1|\}\|x\| \|y\| ), 
\end{equation*}
for any $x, y, z \in \mathcal{H}$ with $\|z\|=1.$ The value $\alpha=2$ gives Buzano  inequality.

We note that the inequality \eqref{alphabuz} was previously obtained by Moslehian et al. (\cite{MKD}, Corollary 2.5) using properties of  singular values.


The main idea in the previous proof  was to obtain a bound for the distance between a rank one operator and the identity operator. On the other hand, Fujii and Kubo in \cite{fujiikubo} based their proof of Buzano  inequality on the fact that $\|2P-I\|\leq 1$ where $P$ is an orthogonal projection. Because of the above, we are in a position to prove our first result in this paper, which generalizes these previous ideas.

\begin{proposition}\label{generalizacion}
	Let $T\in  \mathcal{B}(\mathcal{H})$ and $\alpha \in \mathbb{C}-\{0\}$,  with $\left\|\alpha T-I\right\|\leq1$. Then, for any $x, y \in \mathcal{H}$
	\begin{equation*}
		\left |\langle Tx, y\rangle-\frac{1}{\alpha} \langle x, y\rangle\right|\leq \frac{1} {|\alpha|}\|x\| \|y\|, 
	\end{equation*}
	and 
	\begin{equation}\label{buzanogeneral}
		|\langle Tx, y\rangle| \leq\left |\langle Tx, y\rangle-\frac{1}{\alpha} \langle x, y\rangle\right|+ \frac{1}{|\alpha|} |\langle x, y\rangle|\leq  \frac{1}{|\alpha|}(|\langle x, y\rangle|+\|x\| \|y\| ).
	\end{equation}
	On the other, if $T$ fulfills 
	\begin{equation}\label{eq22}
		|\langle Tx, y\rangle| \leq \frac{1}{|\alpha|}(|\langle x, y\rangle|+\|x\| \|y\| ),
	\end{equation}
	for any $x,y\in \mathcal{H}$ and for some $\alpha\in \mathbb{C}-\{0\}$. Then, $dist(\alpha T, \mathbb{C}I)\leq 1.$

\end{proposition}

\begin{proof}
	Let  $x, y\in \mathcal{H}$ and  $\alpha \in \mathbb{C}-\{0\}$,  with $\left\|\alpha T-I\right\|\leq1$. By \eqref{CS}, we have 
	\begin{eqnarray}
		\left|\langle Tx, y\rangle -\frac{1} {\alpha} \langle x, y\rangle \right|&=&\left|\left\langle \left(T-\frac{1} {\alpha} I\right)x, y\right\rangle \right| 
		\leq \frac{1}{|\alpha|} \left\|\alpha T-I\right\| \|x\| \|y\| \:\:\: \nonumber\\
		&\leq& \frac{1}{|\alpha|} \|x\| \|y\|.
		\nonumber\
	\end{eqnarray}
	Therefore, we obtain 
	\begin{eqnarray}
		|\langle Tx, y\rangle| \leq\left |\langle Tx, y\rangle-\frac{1}{\alpha} \langle x, y\rangle\right|+ \frac{1}{|\alpha|} |\langle x, y\rangle|\leq  \frac{1}{|\alpha|}(|\langle x, y\rangle|+\|x\| \|y\| ). \nonumber
	\end{eqnarray}
	If $T$ satisfies \eqref{eq22} and we recall the formula which express the distance from $T$ to the one-dimensional
	subspace $\mathbb{C}I$ (see \cite{BS}),
	$$dist(T, \mathbb{C}I) = \sup\{|\langle Tx, y\rangle| : \|x\|=\|y\|=1, \langle x, y\rangle = 0\},$$
	then  there exists $\beta\in \mathbb{C}-\{0\}$ such that $\left\|\alpha T-\beta I\right\|\leq1$.
\end{proof}

The next example shows that not all bounded linear operator satisfies the hypothesis
of Proposition \ref{generalizacion}.

\begin{example}
	Consider the unilateral shift operator, $T:l^2({\mathbb{N}})\to l^2({\mathbb{N}}) $, defined by 
	$T(x_1, x_2, x_3, \cdots, )=(0, x_1, x_2, x_3, \cdots )$. Let $e_1=(1, 0, 0, \cdots)\in l^2({\mathbb{N}})$, then $\|e_1\|=1$ and $\langle -Te_1, e_1\rangle=0$, by Theorem 2.1 in \cite{BB}, we have that for any $\alpha\in \mathbb{C}-\{0\}$ it holds
	$$
	\|I\|^2+|\alpha|^2 m^2(-T)\leq \|I-\alpha T\|^2.
	$$
	
	As $-T$ is a left invertible operator, then $m(-T)>0$ and in consequence 
	$$
	1<\|I\|^2+|\alpha|^2 m^2(-T)\leq \|I-\alpha T\|^2.
	$$
\end{example}


For convenience, for any $\alpha \in \mathbb{C}-\{0\}$, we denote by 
$$
\mathcal{A_{\alpha}}=\{T\in \bh:  \|\alpha T-I\|\leq 1\},
$$
the set of bounded operators that fulfills the hypothesis of Proposition \ref{generalizacion}. Since we have proved that each operator belonging to the set $\mathcal{A_{\alpha}}$ satisfies a $\frac{1}{\alpha}$ Buzano-type inequality, we will call to such  set the $\frac{1}{\alpha}$ Buzano set.

Next, we collect some properties of $\mathcal{A_{\alpha}}.$

\begin{proposition}\label{propiedad}
	Let $\alpha\in \mathbb{C}-\{0\}$, then
	\begin{enumerate}
		\item $\mathcal{A_{\alpha}}$ is a non-empty convex and closed set. 
		
	\end{enumerate}
	For any $T\in  \mathcal{A_{\alpha}}$, then
	
	\begin{enumerate}
		\setcounter{enumi}{1}
		\item $\|T\|\leq \frac{2}{|\alpha|}.$
		\item $T^* \in\mathcal{A_{\overline{\alpha}}}$.
		\item If  $S\in 	 \mathcal{A_{\alpha}}$, then $T+S \in  \mathcal{A}_{\frac{\alpha}{2}}$. 
		\item If $\left\|\alpha T-I\right\|<1$, then $T\in \mathcal{GL}(\mathcal{H})$.
		\item If $T$ is self-adjoint, then $T\geq 0$ or $-T\geq 0$.
		\item If $T=h\otimes h$, $h\in \h$ and $h\neq 0$, then $\alpha=\frac{t e^{i\theta}+1}{\|h\|^2}$ for every $\theta\in[0,2\pi]$ and $t\in[-1,1]$.
		\item $d(T,\mathbb{C} I)\leq \frac{1}{|\alpha|}$ and for every $P\geq 0$ with $tr(P)=1$
		$$tr(|T|^2P)-|tr(TP)|^2\leq \frac{1}{|\alpha|}.$$
		
	\end{enumerate} 
\end{proposition}

\begin{proof}
	
	(1) Let $T,S \in \mathcal{A_{\alpha}}$ and $\lambda \in [0,1]$, then 
	\begin{eqnarray}
		\|\alpha(\lambda T+(1-\lambda)S)-I\|&\leq& \|\alpha \lambda T-\lambda I\|+\|\alpha (1-\lambda)S-(1-\lambda) I\|\nonumber \\
		&=&\lambda\|\alpha  T- I\|+(1-\lambda)\|\alpha S- I\|\nonumber\\
		&\leq& \lambda+1-\lambda=1. \nonumber \
	\end{eqnarray}
	This shows that $\lambda T+(1-\lambda)S \in \mathcal{A_{\alpha}}$ and therefore, $\mathcal{A_{\alpha}}$ is convex.
	
	Now, let  $\{T_n\}$ be a sequence in $\mathcal{A_{\alpha}}$ such  that converges to $T\in \bh$. We must show that $T\in \mathcal{A_{\alpha}}$. Then, we have for any $n\in \mathbb N$ that it hold 
	\begin{eqnarray}
		\|\alpha T-I\|=\|\alpha T-\alpha T_n+\alpha T_n-I\|\leq |\alpha|\|T_n-T\|+1. \nonumber \
	\end{eqnarray}
	Taking limit when $n$ tends to infinity we obtain $\|\alpha T-I\|\leq 1$, i.e. $T\in \mathcal{A_{\alpha}}$. 
	
	The proof of items (2), (3), (4), and (5) are trivial.
	
	(6) If $T$ is normal, then $(\alpha T-I)^*(\alpha T-I)=(\alpha T-I)(\alpha T-I)^*$, for every $\alpha\in \C$. Therefore, $\alpha T-I$ is normal and 
	$$r(\alpha T-I)=\omega(\alpha T-I)=\|\alpha T-I\|,$$
	where $r(\alpha T-I)=\sup\{|\beta|: \beta\in \sigma(\alpha T-I) \}$ is the spectral radius. By the spectral theorem 
	$$\|\alpha T-I\|=r(\alpha T-I)=\sup\{|\alpha \lambda-1|: \lambda\in \sigma(T) \}.$$
	Thus, if $T$ is selfadjoint and $T\in\mathcal{A}_{\alpha}$
	$$|\alpha \lambda-1|\leq 1,$$
	for all $\lambda\in \sigma(T)$.
	Therefore, each $\alpha \lambda$ must lie in the unit disk in the complex plane centered in $z=1$. Also,
	$$\lambda\in \R \text{ for every } \lambda\in \sigma(T)\Rightarrow\left\lbrace \begin{array}{ll}
		Re(\lambda\alpha)=\lambda Re(\alpha)\in&[0,2]\\
		Im(\lambda\alpha)=\lambda Im(\alpha)\in&[-1,1]
	\end{array} \right.  $$
	Suppose there exist $\lambda_j,\lambda_k\in \sigma(T)$ such that $\lambda_j<0$ and $\lambda_k>0$, then
	$$Re(\lambda_j\alpha)\in[0,2]\Rightarrow0\geq Re(\alpha)\geq \dfrac{2}{\lambda_j}$$
	and
	$$Re(\lambda_k\alpha)\in[0,2]\Rightarrow0\leq Re(\alpha)\leq \dfrac{2}{\lambda_k}.$$
	Thus, $Re(\alpha)=0$, which means that $\lambda_j\alpha$ is pure imaginary and $\alpha=0$ (because the unit disk in the complex plane centered in $z=1$ intersects the imaginary axis only in $z=0$). This is a contradiction, so $\lambda_j$ and $\lambda_k$ have the same sign.

	(7) For any non-zero $h\in \h$, $T=h\otimes h$  is a  compact, positive operator and its spectrum has only two eigenvalues, $\|h\|^2$ and $0$. Then, by the proof of item (6)
	$$\left| \alpha \|h\|^2-1 \right| \leq1$$
	if and only if $\alpha\|h\|^2=te^{i\theta}+1$, with $|t|\leq 1$ and $\theta\in [0,2\pi] $.
	
	(8) $1\geq \|\alpha T-I\|=|\alpha|\|T-\frac{1}{\alpha}I\|\geq |\alpha|d(T,\mathbb{C} I)\Rightarrow \frac{1}{|\alpha|}\geq d(T,\mathbb{C} I)$ and using Proposition 3.1 in \cite{BC} we obtain that 
	$$tr(|T|^2P)-|tr(TP)|^2\leq \frac{1}{|\alpha|}$$
	for every $P\geq 0$ with $tr(P)=1.$

\end{proof}

Next, we enumerate other properties of $\mathcal{A_{\alpha}}$ related to $dist(I, \mathbb{C}T)$ and the center of mass of $T$.
\begin{lemma} Let $T\in \mathcal{B}(\mathcal{H})$.
	\begin{enumerate}
		\item If $m(T)>0$ and $c(T)\neq0$, then $T\in \mathcal{A}_{c(T)}.$ 
		\item If $m(T)>0$ and $c(T)=0$, then $I$ is Birkhoff-James orthogonal to $T$. We deduce that $1<\|\alpha T-I\|$  and $T\notin \mathcal{A_{\alpha}}$, for every $\alpha\in \mathbb{C}-\{0\}$.	
		\item If  $dist(I, \mathbb{C}T)=1$ and there exists $\alpha_0 \in \mathbb{C}-\{0\}$ such that 
		$$1=\|\alpha_0T-I\|<\|\alpha T-I\|, \text{ for all } \alpha\in \C-\{\alpha_0,0\},$$
		then $0\in \sigma_{app}(T)$ and $T\in \mathcal{A}_{\alpha_0}$.
	\end{enumerate}
\end{lemma}


\begin{proposition} 
	Let $T\in \bh$ such that $m(T)>0$, $c(T)\neq 0$ and  $dist(T,\mathbb CI)=\|T\|,$ then $\|I-c(T)T\|=1$ and, in particular, $T\in \mathcal{A}_{c(T)}$.
\end{proposition}

\begin{proof}
	Let $x\in \mathcal{H}$ such that $\|x\|=1$, then $|\langle Tx, x\rangle|\geq \|Tx\|\inf_{\|y\|=1}\frac{|\langle Ty, y\rangle|}{\|Ty\|}$ and 
	$$
	\left[\|Tx\|^2-|\langle Tx, x\rangle|^2 \right]^{1/2}\leq \left(1-\inf_{\|y\|=1}\frac{|\langle Ty, y\rangle|^2}{\|Ty\|^2}\right)^{1/2}\|Tx\|.
	$$
	Calculating the supremum of both sides, and using the equality \eqref{paul}, we get
	$$
	\|T\|=dist(T,\mathbb CI)\leq dist(I,\mathbb CT)\|T\|.
	$$
	Then $1\leq  dist(I,\mathbb{C}T) \leq \|I\|=1$, i.e. $dist(I,\mathbb{C}T)=\|I-c(T)T\|=1$. This completes the proof.
\end{proof}
As we have shown in Proposition \ref{propiedad}, if $T\in \mathcal{A_{\alpha}}$ with $\|\alpha T-I\|<1$, then $T\in \mathcal{GL}(\mathcal{H})$ and $T$ verifies Proposition \ref{generalizacion}. Now, we obtain a generalization of such statement for any invertible operator. In order to prove it, we need the following result.

\begin{lemma}[Corollary 3.7, \cite{BD}]\label{Bracic-Diogo}
	If $T\in \mathcal{GL}(\mathcal{H})$, then there exists a unitary operator $U\in \bh$ and a non-zero complex number $\beta$ such that $\|\beta T^{-1}-U^*\|<1.$
\end{lemma}
Observe that in the value $\beta$ in Lemma \ref{Bracic-Diogo} satisfies that $|\beta|< \frac{2}{\|T^{-1}\|}$.

\begin{theorem}\label{inv}
	Let $T\in \mathcal{GL}(\mathcal{H})$, then there exists a unitary operator $U\in \bh$ and a non-
	zero complex number $\beta$ as in Lemma \ref{Bracic-Diogo} such that  for any $x, y \in \mathcal{H}$
	\begin{equation*}
		\left |\langle T^{-1}x, y\rangle-\frac{1}{\beta} \langle U^*x, y\rangle\right|\leq \frac{1} {|\beta|}\|x\| \|y\|, 
	\end{equation*}
	and 
	\begin{equation*}
		|\langle T^{-1}x, y\rangle| \leq\left |\langle T^{-1}x, y\rangle-\frac{1}{\beta} \langle U^*x, y\rangle\right|+ \frac{1}{|\beta|} |\langle U^* x, y\rangle|\leq  \frac{1}{|\beta|}(|\langle U^* x, y\rangle|+\|x\| \|y\| ).
	\end{equation*}
\end{theorem}
\begin{proof}
	It is analogous to the proof of Proposition \ref{generalizacion}, so we omit it.
\end{proof}

\subsection{Bounded Linear operators which belong to $\mathcal{A}_{\alpha}$}\label{s4}

We begin this subsection showing when a normal operator belongs to $\mathcal{A_{\alpha}}$.
\begin{theorem}\label{alpha normales}
	Let $T$ be a normal operator in $\bh-\{0\}$, such that $\sigma(T)$ is fully included into an arc of the disk of radius $\|T\|$ and centered in the origin, with central angle less than $\pi$. \\	
	Then, $T\in \mathcal{A}_{\alpha}$ for every $\alpha\in \C$ such that
	\begin{equation}\label{normal1}
		\arg(\alpha)+\arg(\lambda)\in\left(-\frac{\pi}{2},\frac{\pi}{2}\right),\ \text{for all } \lambda\in \sigma(T)
	\end{equation}
	and 
	\begin{equation}\label{normal2}
		|\alpha|\leq \frac{2}{\|T\|}\min_{\lambda\in \sigma(T)}\cos(\arg(\alpha)+\arg(\lambda)).
	\end{equation}
\end{theorem}
\begin{proof}
	Since $T$ is normal, $\alpha T-I$ is a normal operator, then
	$$\|\alpha T-I\|=r(\alpha T-I)=\sup\{|\alpha \lambda-1|: \lambda\in \sigma(T) \}\leq 1$$
	if and only if $|\alpha \lambda-1|\leq 1$ for every $\lambda\in \sigma(T)$.
	This is equivalent to find if there exists $\alpha\in \C$ such that 
	\begin{equation}\label{cond less 1}
		|\alpha \lambda-1|^2\leq 1\ \text{ for every } \lambda \in\sigma(T).
	\end{equation}
	Taking $\alpha=|\alpha| e^{i\theta}$ and $\lambda=|\lambda|e^{i\varphi_{\lambda}}$, with $|\lambda|\leq \|T\|$, 
	\begin{eqnarray*}
		\left| |\alpha||\lambda|e^{i(\theta+\varphi_{\lambda})}-1\right|^2&=&\left(|\alpha||\lambda|\cos(\theta+\varphi_{\lambda})-1 \right) ^2+\left(|\alpha||\lambda| \sin(\theta+\varphi_{\lambda})\right)^2\\
		&=&|\alpha|^2|\lambda|^2-2|\alpha||\lambda|\cos(\theta+\varphi_{\lambda})+1
	\end{eqnarray*}
	and we can rewrite \eqref{cond less 1} as follows
	\begin{eqnarray*}
		|\alpha||\lambda|\left(|\alpha||\lambda|-2\cos(\theta+\varphi_{\lambda})\right) &\leq& 0.\\
	\end{eqnarray*}
	Then, we arrive to the following condition
	\begin{equation}\label{cond alpha}
		|\alpha||\lambda|-2\cos(\theta+\varphi_{\lambda}) \leq 0. 
	\end{equation}
	Take an $\alpha\in \C$ that satisfies \eqref{normal1} and \eqref{normal2}.
	For $\lambda=0$ it is immediate that $\alpha$ satisfies condition \eqref{cond alpha}. Consider $\lambda\neq 0$, then
	$\cos(\arg(\alpha)+\varphi_{\lambda}) > 0$  for every $\lambda\in \sigma(T)$ and
	$$|\alpha|\leq \frac{2}{\|T\|}\min_{\lambda\in \sigma(T)}\cos(\arg(\alpha)+\arg(\lambda))
	\leq \frac{2}{|\lambda|}\cos(\arg(\alpha)+\arg(\lambda)), \lambda\neq 0.$$
	Therefore, $\alpha$ fulfills the condition \eqref{cond alpha} for every $\lambda\in \sigma(T)$ and we conclude that $\|\alpha T-I\|\leq 1$ ($T\in \mathcal{A}_{\alpha}$).
\end{proof}

In order to fulfill \eqref{cond alpha} and  $\cos(\arg(\alpha)+\varphi_{\lambda})\geq 0$,  it is a necessary condition that the spectrum of $T$ lies in into an arc of the disk of radius $\|T\|$ and centered in the origin, with central angle less than $\pi$. Otherwise, it is not possible to fix any $\alpha$ such that the property holds. Additionally, we exclude $\arg(\alpha)+\varphi_{\lambda}=\pm \frac{\pi}{2}$, since $|\alpha||\lambda|= 0$  if and only if $\lambda=0$ or $\alpha=0$.

For example, if $T$ is Hermitian $\varphi_{\lambda}\in \{0,\pi\}$ it can be seen that there is no $\alpha$ that \eqref{cond alpha} holds,  unless $\lambda\geq 0$ for every $\lambda\in \sigma(T)$ ($\varphi_{\lambda}=0$), or $\lambda\leq 0$ for every $\lambda\in \sigma(T)$ ($\varphi_{\lambda}=\pi$). Thus, the unique Hermitian operators $T$ that can reach \eqref{cond alpha} are semidefinite positive or semidefinite negative, as we show in item 6 of Proposition \ref{propiedad}.

In particular, for positive operators, we arrive to the following result.
\begin{corollary}\label{positivosgral}
	If $T\in \mathcal{B}(\mathcal{H})^+$, then $T\in \mathcal{A}_{\alpha}$ for every $\alpha\in \C$ such that $|\alpha|\leq \frac{a}{\|T\|}\leq \frac{2}{\|T\|}$ and $\cos(\arg(\alpha))\geq \frac{a}{2}$.	
\end{corollary}
\begin{proof}
	As we mention before, in this case $\varphi_{\lambda}=0$ for every $\lambda\in \sigma(T)$. Then,
	$$|\alpha||\lambda|-2\cos(\arg(\alpha))\leq \frac{a}{\|T\|}|\lambda|-2\cos(\arg(\alpha))\leq a-2\cos(\arg(\alpha))\leq 0. $$ 
\end{proof}

The next result is a generalization of Buzano inequality for any bounded linear operator.

\begin{theorem}\label{positivo}
	Let $T \in \bh-\{0\}.$ Then, for any $x, y\in\h$
	\begin{equation*}
		\left |\langle Tx, Ty\rangle-\frac{\|T\|^2}{2} \langle x, y\rangle\right|\leq \frac{\|T\|^2}{2}\|x\| \|y\|, 
	\end{equation*}
	and  
	\begin{equation}\label{T^*T}
		\left|\left\langle Tx, Ty\right\rangle\right|\leq \left |\langle Tx, Ty\rangle-\frac{\|T\|^2}{2} \langle x, y\rangle\right|+ \frac{\|T\|^2}{2} |\langle x, y\rangle| \leq \frac{\|T\|^2}{2}(|\langle x, y\rangle|+\|x\| \|y\| ).
	\end{equation}
\end{theorem}
\begin{proof}
	By Corollary \ref{positivosgral}, if $S\in \bh^+-\{0\},$
	then $S\in \mathcal{A}_{\frac{2}{\|S\|}}$.  In particular, if we consider the positive operator $S=T^*T$, then we conclude that $T^*T\in \mathcal{A}_{\frac{2}{\|T\|^2}}$ and the proof is complete as a consequence of  Proposition \ref{generalizacion} 
\end{proof}
The constant $\frac{\|T\|^2}{2}$ is best possible in \eqref{T^*T}. 
Now, if we assume that \eqref{T^*T} holds with a constant $C>0$, i.e.
\begin{equation*}\label{dragomirC}
	\left|\left\langle Tx, Ty\right\rangle\right|\ \leq C(|\langle x, y\rangle|+\|x\| \|y\| ),
\end{equation*}
for any $T\in \bh$. So, if we choose $x=y$, then 
$\|Tx\|^2 \leq 2C\|x\|^2$ and we deduce that $2C\geq \|T\|^2.$ Thus, \eqref{T^*T} is an improvement and refinement of
\begin{equation*}\label{Dragomirgeneral}
	\left|\left\langle Tx, Ty\right\rangle\right|\leq  \frac{\|T\|^2}{2}(|\langle x, y\rangle|+\|x\| \|y\| ),
\end{equation*}
which was obtained in a different way earlier by Dragomir in \cite{Dra17} using a non-negative Hermitian form on a Hilbert space. 

From the polar decomposition of any bounded linear operator and the main idea used in the proof of Theorem \ref{positivo}, we have the following statement.
\begin{corollary}
	Let $T\in  \mathcal{B}(\mathcal{H})$ and $x, y \in \mathcal{H}$. Then, 
	\begin{eqnarray}\label{buzanogeneral2}
		|\langle Tx, y\rangle|&=&	
		|\langle |T|x, V^* y\rangle|\leq
		\left |\langle Tx, y\rangle-\frac{\|T\|}{2} \langle x, V^*y\rangle\right|+ \frac{\|T\|}{2}  |\langle x, V^* y\rangle|\nonumber \\
		&\leq&   \frac{\|T\|}{2}(|\langle x, V^*y\rangle|+\|x\| \|V^*y\| )\nonumber \\ 
		&\leq& \frac{\|T\|}{2}(|\langle x, V^*y\rangle|+\|x\| \|y\| ).
	\end{eqnarray}
	where $T=V|T|$ is the polar decomposition of $T$.
\end{corollary}
\begin{remark}
	Inequality	\eqref{buzanogeneral2} is an improvement and refinement of a result recently obtained by Sababheh et al. (see \cite[Remark 3.1]{SMH}).
\end{remark}

\bigskip 


Recall that $T$ is called a positive contraction if $0\leq T\leq I.$ As a consequence of Corollary \ref{positivosgral}, we conclude that $T\in \mathcal{A}_{\alpha}$ for every $\alpha\in [0,2]$.

\bigskip 
Now, we obtain a refinement of the classical Cauchy-Schwarz  inequality, using positive contractions. The idea of the proof is based in \cite[Theorem 2.1]{Dra17}. Recently, in \cite{SMH} the same result was obtained with a different proof.

\begin{theorem}\label{refCScontraction}
	Let $T\in \bh$  be a positive contraction and $x, y \in \mathcal{H},$ then
	\begin{equation*}
		|\langle x, y\rangle|+\langle Tx, x\rangle^{1/2}\langle Ty, y\rangle^{1/2}-|\langle Tx, y\rangle|\leq \|x\| \|y\|.
	\end{equation*}
\end{theorem}
\begin{proof}
	For any $x,y \in \mathcal{H}$ and the elementary  inequality $(ac-bd)^2\geq (a^2-b^2)(c^2-d^2)$, which holds for any real
	numbers $a, b, c, d,$
	we have
	\begin{eqnarray}\label{ineqa}
		\left(\|x\| \|y\|-\langle Tx, x\rangle^{1/2}\langle Ty, y\rangle^{1/2}\right)^2&\geq& (\|x\|^2-\langle Tx, x\rangle)(\|y\|^2-\langle Ty, y\rangle)\nonumber\\
		&=&\langle(I-T)x, x\rangle \langle(I-T)y, y\rangle. \
	\end{eqnarray}
	As $T$ is a positive contraction, then $I-T\in \bh^+.$ By Cauchy-Schwarz  inequality for positive operators 
	\begin{equation}\label{CSpositive}
		\langle(I-T)x, x\rangle \langle(I-T)y, y\rangle\geq |\langle(I-T)x, y\rangle|^2=|\langle x, y\rangle-\langle Tx, y\rangle |^2.
	\end{equation}
	Now, by \eqref{ineqa} and \eqref{CSpositive},
	\begin{equation}\label{ineqb}
		\left(\|x\| \|y\|-\langle Tx, x\rangle^{1/2}\langle Ty, y\rangle^{1/2}\right)^2\geq |\langle x, y\rangle-\langle Tx, y\rangle |^2, 
	\end{equation}
	for any $x, y \in \mathcal{H}.$ Since  $\|x\|\geq \langle Tx, x\rangle^{1/2}$ and $\|y\|\geq \langle Ty, y\rangle^{1/2}$, by taking the square root, \eqref{ineqb} is equivalent to
	\begin{equation}\label{ineqc}
		\|x\| \|y\|-\langle Tx, x\rangle^{1/2}\langle Ty, y\rangle^{1/2}\geq |\langle x, y\rangle-\langle Tx, y\rangle |. 
	\end{equation}
	On making use of  the triangle inequality for the
	modulus,
	we have 
	\begin{eqnarray}\label{ineqd}
		\|x\| \|y\|-\langle Tx, x\rangle^{1/2}\langle Ty, y\rangle^{1/2}&\geq& |\langle x, y\rangle-\langle Tx, y\rangle |\nonumber \\
		&\geq& |\langle x, y\rangle|-|\langle Tx, y\rangle |,
	\end{eqnarray}
	and this completes the proof. 
\end{proof}

\begin{remark}
	Recall that any orthogonal projection $P=P^2=P^*$ is a positive contraction with $\|P\|=1$ and $P\in \mathcal{A}_2$. Then, for any $x,y\in \mathcal{H}$ 
	\begin{equation}\label{proy1}
		|\langle Px, y\rangle| \leq \left|\langle Px, y\rangle-\frac12 \langle x, y\rangle\right|+\frac12|\langle x, y\rangle|\leq \frac12(|\langle x, y\rangle|+\|x\| \|y\| ),
	\end{equation}
	and
	\begin{equation}\label{proy2}
		\left|\langle Px, y\rangle-\langle x, y\rangle\right|\leq \frac12(|\langle x, y\rangle|+\|x\| \|y\| ).
	\end{equation}
	In \eqref{proy1} and \eqref{proy2}  we reach,  with a new proof,  an improvement and refinement of different statements previously obtained by Dragomir in \cite{Dra16}.
	
	Motivated by the previous inequalities valid for orthogonal projections, we establish some vector inequalities 
	for particular  projections.
	Let $T=z\otimes z$, with $z\in \mathcal{H}$ and  $\|z\|=1$,  as $T$ is an orthogonal projection then by 
	using inequality \eqref{buzanogeneral} we get
	\begin{equation*}\label{refdragomir}
		|\langle x, z\rangle \langle z, y\rangle| \leq \left|	\langle x, z\rangle \langle z, y\rangle-\frac12 \langle x, y\rangle\right|+\frac12 |\langle x, y\rangle|
		\leq \frac12(|\langle x, y\rangle|+\|x\| \|y\| ),
	\end{equation*}
	for any $x, y \in \mathcal{H}.$ This inequality refines the classical Buzano inequality.\\
	
	On the other hand,  using inequality \eqref{ineqc} 
	\begin{equation*}
		|\langle x,y\rangle|\leq |\langle x,y\rangle-\langle x,z\rangle\langle z,y\rangle|+|\langle x,z\rangle\langle z,y\rangle|\leq \|x\|\|y\|,
	\end{equation*}
	for any $x,y \in \mathcal{H}.$ This refinement of \eqref{CS} was also obtained in \cite{Dra85}.
\end{remark}

Now we are in position to obtain a Buzano type inequality for the sum of two orthogonal projections. It is well-known that given two orthogonal projections on $\mathcal{H}$, $P$ and $ Q$,  then 
\begin{equation}\label{Duncan_Taylor_Kittaneh}
	\|P+Q\|= 1+\|PQ\|.
\end{equation}
This is usually called Duncan-Taylor equality, and its proof can be found in \cite{DT}.

\begin{proposition}\label{sumproj}
	Let $P, Q$  be  orthogonal projections on $\mathcal{H}$. 
	Then, $P+Q\in \mathcal{A_{\alpha}}$ for $|\alpha|\leq \frac{2}{1+\|PQ\|}$.
	
\end{proposition}
\begin{proof}
	Note that $P+Q\in \bh^+$ and by \eqref{Duncan_Taylor_Kittaneh}, $\|P+Q\|= 1+\|PQ\|$.
	Thus, using Corollary \ref{positivosgral}, the proof is complete.
\end{proof}

Throughout, $\mathcal{S}$ and $\mathcal{T}$ denote two closed subspaces of  $\cH$. The mininal angle or 
angle of Dixmier between $\mathcal{S}$ and $\mathcal{T}$ 
is the angle $\theta_0(\mathcal{S}, \mathcal{T})\in [0, \frac{\pi}{2}]$ 
whose cosine is defined by
\begin{equation*}
	c_0(\mathcal{S}, \mathcal{T})=\sup\{|\langle x, y\rangle|: x\in \mathcal{S}, y\in \mathcal{T}; \|x\|, \|y\|\leq 1\}.
\end{equation*}
A linear operator defined on  $\mathcal{H}$, such that $Q^2=Q$ is called a  projection.  Such operators are not necessarily bounded, since on every infinite-dimensional Hilbert space there exist unbounded examples of projections (see \cite{Buc}). The operator $Q_{\mathcal{M}//\mathcal{N}}$ is an oblique projection along (or parallel to) its null space $\mathcal{N} = \mathcal{N}(Q)$ onto
its range $\mathcal{M} = \mathcal{R}(Q)$.


\begin{theorem} 	Let $H$ be a Hilbert space such that is the direct
	sum of closed subspaces $\mathcal{M}$ and $\mathcal{N}$. Let $Q_{\mathcal{M}//\mathcal{N}}$, be the bounded projection with range $\mathcal{M}$ and null space $\mathcal{N}$, and $\theta_0(\mathcal{M}, \mathcal{N})$, be the minimal angle between $\mathcal{M}$
	and $\mathcal{N}$. Then, for any $x, y \in \mathcal{H}$
	\begin{equation*}
		\left |\langle Qx, y\rangle-\frac{1}{2} \langle x, y\rangle\right|\leq \frac{\cot(\alpha_0)} {2}\|x\| \|y\|, 
	\end{equation*}
	and 
	\begin{equation*}
		|\langle Qx, y\rangle| \leq \frac{\cot(\alpha_0)} {2}(|\langle x, y\rangle|+\|x\| \|y\| ), 
	\end{equation*}
	where $\alpha_0=\frac{\theta_0(\mathcal{M}, \mathcal{N})}{2}.$
\end{theorem}

\begin{proof}
	By Theorem 2 in \cite{Buc} we have that $\|Q\|=\csc(\theta_0(\mathcal{M}, \mathcal{N}))$ and $\|2Q-I\|=\cot(\alpha_0).$ From the boundness of $Q$ we can assert that $0<\theta_0(\mathcal{M}, \mathcal{N})\leq \frac{\pi}{2}$ and  $\cot(\alpha_0)\geq 1.$ From these facts and mimicking the proof of Proposition \ref{generalizacion}, we have that for any $x, y \in \cH$
	\begin{equation*}
		\left |\langle Qx, y\rangle-\frac{1}{2} \langle x, y\rangle\right|\leq \frac{\cot(\alpha_0)} {2}\|x\| \|y\|, 
	\end{equation*}
	and 
	\begin{eqnarray}\label{buzanogeneraloblique2}
		|\langle Qx, y\rangle| &\leq&\left |\langle Qx, y\rangle-\frac{1}{2} \langle x, y\rangle\right|+ \frac{1}{2} |\langle x, y\rangle|\nonumber\\
		&\leq& \frac{\cot(\alpha_0)} {2}\|x\| \|y\|+
		\frac{1}{2} |\langle x, y\rangle|\nonumber\\ 
		&\leq& \frac{\cot(\alpha_0)} {2}(|\langle x, y\rangle|+\|x\|\|y\|). \nonumber\
	\end{eqnarray}
	This completes the proof. 
\end{proof}

We finish this section by showing that any operator whose real part is greater than $sI$ for some $s>0$, is invertible and its inverse belongs to $\mathcal{A}_{2s}.$ 

\begin{theorem}
	Let $T \in \bh-\{0\}$ with $Re(T)=\frac{T+T^*}{2}\geq sI$ for some $s>0.$ Then, $T^{-1}\in \mathcal{A}_{2s}.$
\end{theorem}
\begin{proof}
	First, we show that $T$ is invertible. The hypothesis $Re(T)\geq sI$ implies that  
	$$
	W(T)\subseteq \{z\in \mathbb{C}:Re(z)\geq s\},
	$$
	since if $z\in W(T)$ then 
	\begin{eqnarray}
		Re(z)&=&\frac{z+\overline{z}}{2}=\frac{\langle Tx, x\rangle +\overline{\langle Tx, x\rangle}}{2}=\frac{\langle Tx, x\rangle +\langle T^*x, x\rangle}{2}\nonumber\\
		&=&\langle Re(T)x, x\rangle\geq s.\nonumber\
	\end{eqnarray}
	Thus $\sigma(T)\subseteq \overline{W(T)}\subseteq \{z\in \mathbb{C}:Re(z)\geq s\}$ and, in particular, we have that $0\notin \sigma(T),$ which means $T\in \mathcal{GL}(\mathcal{H}).$
	If $T+T^*\geq 2sI$ and $T\in \mathcal{GL}(\mathcal{H})$, then $2sT^{-1}\left(T+T^*- 2sI\right)(T^*)^{-1}\geq 0$ and
	$$
	I\geq I-2sT^{-1}\left(T+T^*- 2sI\right)(T^*)^{-1}=(I-2sT^{-1})(I-2sT^{-1})^*,
	$$
	which is equivalent to $\|2sT^{-1}-I\|\leq 1.$ Hence, $T^{-1}\in \mathcal{A}_{2s}$ and the result follows.
\end{proof}

\section{Bounds for the numerical radius using Buzano inequality}\label{s5}

In this section, we use  \eqref{buzanogeneral2} to obtain a refinement of the classical inequality $\omega(T)\leq \|T\|$ and an upper bound for $\omega(T)-\frac{1}{2}\|T\|.$

\begin{proposition}
	Let $T\in  \mathcal{B}(\mathcal{H})$ with polar decomposition $T=V|T|$. Then, 
	\begin{equation}\label{polar1}
		\omega(T)\leq \frac{\|T\|}{2}\left(1+\omega(V)\right)\leq \|T\|
	\end{equation}
	and
	\begin{equation}\label{polar2}
		\omega(T)- \frac{\|T\|}{2}\leq \frac{\|T\|}{2}\omega(V).
	\end{equation}
	
\end{proposition}

\begin{proof}
	Taking $x=y$ in \eqref{buzanogeneral2}, and the supremum over all $x\in \cH$ with $\|x\|=1$, we obtain $\omega(T)\leq \frac{\|T\|}{2}\left(1+\omega(V)\right)$ and this completes the proof. 
	
\end{proof}

It is important to note that inequalities \eqref{polar1} and \eqref{polar2} are not trivial, since $\omega(V)$ may be less than one, depending on the partial isometry $V$. For instance, let
$$T=\begin{bmatrix}
	0&0\\
	\sqrt{2}&0
\end{bmatrix}=\begin{bmatrix}
	0&0\\
	1&0
\end{bmatrix}\begin{bmatrix}
	\sqrt{2}&0\\
	0&0
\end{bmatrix}=V|T|,$$
where $V$ is a partial isometry in $\C^2$ with $\ker(T)=\ker(V)=\text{span}\{(0,1)\}$ and $\ker(V)^{\perp}=\text{span}\{(1,0)\}$. Then, for any $x=(x_1,x_2)\in \C$ with $\|x\|=\sqrt{|x_1|^2+|x_2|^2}=1$, we have as a consequence of the arithmetic-geometric mean inequality
$$\left| \left\langle Vx,x \right\rangle\right| = |x_1\overline{x_2}|=|x_1||x_2|\leq \dfrac{|x_1|^2+|x_2|^2}{2}=\frac{1}{2}.$$
Therefore, $W(V)=\{z\in \C: |z|\leq \frac{1}{2}\}$ and $\omega(V)=\frac{1}{2}<1$.

\begin{proposition}\label{omegaigualnorma} 
	Let $T\in  \mathcal{B}(\mathcal{H})- \{0\}$ with polar decomposition $T=V|T|$ such that $\omega(T)=\|T\|$, then $\omega(V)=\|V\|=1.$
\end{proposition}
\begin{proof} From inequality \eqref{polar1}, we get 
	$$\omega(T)\leq \frac{\|T\|}{2}\left(1+\omega(V)\right)\leq \|T\|.
	$$ Thus, if $\omega(T)=\|T\|$, then $1+\omega(V)=2,$ and hence $\omega(V)=1$. As $V$ is a nonzero partial isometry, therefore $\|V\|=1$ and thus $\omega(V)=\|V\|=1,$ as required.
\end{proof}

It should also be mentioned here that the converse of Proposition \ref{omegaigualnorma} is not
true. To see this, consider
$$
T=\begin{bmatrix}
	0 & 0 & 0\\
	1 & 0 & 0\\
	0 & 1 & 0
\end{bmatrix}=\begin{bmatrix}
	0 & 0 & 0\\
	1 & 0 & 0\\
	0 & 1 & 0
\end{bmatrix}\begin{bmatrix}
	1 & 0 & 0\\
	0 & 1 & 0\\
	0 & 0& 0
\end{bmatrix}=V|T|.
$$
where $V|T|$ is a polar decomposition of $T$. Then, $\omega(V)=\|V\|=1$, but $\omega(T)=\frac{1}{\sqrt{2}}<1=\|T\|.$ 

In order to estimate how close the numerical radius is from the operator norm, the following reverse inequalities have been obtained under appropriate conditions for the involved operator $T\in \bh$. If $T\in  \mathcal{A_{\alpha}}$,  then by \eqref{omegaequiv} and  Proposition \ref{propiedad} we have that
\begin{equation*}
	0\leq \|T\|-\omega(T)\leq \|T\|-\frac{\|T\|}{2}\leq \frac{1}{|\alpha|}.
\end{equation*}
Motivated by the above inequality, we establish a new upper bound
for the non-negative quantity $\|T\|-\omega(T)$.

\begin{theorem}
	Let $T\in   \mathcal{A_{\alpha}}$. Then, 
	\begin{equation*}
		\|T\|-\omega(T)\leq \frac{1}{2|\alpha|}.
	\end{equation*}
\end{theorem}
\begin{proof}
	For $x\in \h$ with $\|x\|= 1$, we have
	\begin{equation*}
		\|\alpha Tx-x\|^2=|\alpha|^2\|Tx\|^2-2Re(\alpha \langle Tx, x\rangle)+1\leq 1
	\end{equation*}
	giving 
	\begin{equation*}
		|\alpha|^2\|Tx\|^2+1\leq 1 +2Re(\alpha \langle Tx, x\rangle)\leq 2|\alpha| |\langle Tx, x\rangle|+1.
	\end{equation*}
	By arithmetic-geometric mean inequality, we deduce
	\begin{equation}\label{desigualdad}
		2|\alpha| \|Tx\|\leq |\alpha|^2\|Tx\|^2+1\leq 1 +2Re(\alpha \langle Tx, x\rangle)\leq 2|\alpha| |\langle Tx, x\rangle|+1.
	\end{equation}
	Now, taking the supremum over $x\in \h$, $\|x\|=1$ in \eqref{desigualdad}, we obtain
	\begin{equation*}
		2|\alpha| \|T\|-2|\alpha|\omega(T)\leq 1.
	\end{equation*}
\end{proof}

Now, we derive upper bounds for the numerical radius of products of bounded linear operators.
\begin{theorem}\label{RST}
	Let $R, S, T\in \mathcal{B}(\mathcal{H})$ such that $T\in  \mathcal{A_{\alpha}}$. Then,
	\begin{align}
		\omega(STR)\leq \frac{1}{|\alpha|}\left(\|R\|\|S\| + \omega(SR)\right).
	\end{align}
\end{theorem}
\begin{proof}
	From inequality \eqref{buzanogeneral}, we have
	\begin{align*}\label{ineq2}
		|\langle STRx, y\rangle|=|\langle TRx, S^*y\rangle|\leq \frac{1}{|\alpha|}(|\langle Rx, S^*y\rangle|+\|Rx\| \|S^*y\| ).
	\end{align*}
	Taking $y=x$ and the supremum over $x\in \mathcal{H}$ with $\|x\|=1$, yields the desired inequality.
	
\end{proof}

We note that the previous result is a generalization of Theorem 3.6 in \cite{Dra17}.  In particular, for the sum of two orthogonal projections, we obtain the following result. 

\begin{corollary}
	Let $P,Q,R,S\in \mathcal{B}(\mathcal{H})$ with $P, Q$  be  orthogonal projections. Then,
	\begin{equation*}
		\omega(R(P+Q)S)\leq \frac{1+\|PQ\|}{2}\left(\|S\| \|R\| +\omega(RS)\right).
	\end{equation*}
\end{corollary}

\begin{proof}
	As we have already mentioned, $P+Q\in  \mathcal{A}_{\frac{2}{1+\|PQ\|}}.$ Then, the statement is a consequence of Theorem \ref{RST}. 
\end{proof}

On the other hand,  in \cite{Dra07}, Dragomir obtained,  utilizing Buzano’s inequality, the following inequality for the
numerical radius 
\begin{equation*}
	\omega(S)^2\leq \frac{1}{2}\left(\|S\|^2 + \omega(S^2)\right),
\end{equation*}
combining  with the following power inequality for the numerical radius, $w(S^n)\leq w(S)^n$ for any natural number $n$, we have
\begin{equation}\label{Dragomir}
	w(S^2)\leq \frac{1}{2}\left(\|S\|^2 + \omega(S^2)\right).
\end{equation}

The following corollary, which is an immediate consequence of Theorem \ref{RST} considering $R = S$, gives a generalization of \eqref {Dragomir}.

\begin{corollary}
	Let $S, T\in \mathcal{B}(\mathcal{H})$ with $T\in  \mathcal{A_{\alpha}}$. Then,
	\begin{align*}\label{RequalS}
		\omega(STS)\leq \frac{1}{|\alpha|}\left(\|S\|^2 + \omega(S^2)\right).
	\end{align*}
\end{corollary}

%
\section*{Funding}

The research of Dr. Tamara Bottazzi is partially supported by National Scientific and Technical Research Council of Argentina (CONICET), Universidad Nacional de R\'io Negro and ANPCyT PICT 2017-2522.

The research of Dr. Cristian Conde is partially supported by National Scientific and Technical Research Council of Argentina (CONICET), Universidad Nacional de General Sarmiento and ANPCyT PICT 2017-2522.


\begin{thebibliography}{99} 
	\footnotesize
	
	\bibitem{BB} M. Barraa and M. Boumazgour,
	\textit{A note on the orthogonality of bounded linear operators}. Funct. Anal. Approx. Comput. {\bf 4}
	(2012), no. 1, 65--70.
	
	\bibitem{BS} R. Bhatia and  P. \textrm{$\check{S}$}emrl, \textit{Orthogonality of matrices and some distance problems},  Linear Algebra Appl. \textbf{287} (1999), no. 1-3, 77--85.
	
	\bibitem{Buc} D. Buckholtz, \textit{Hilbert space idempotents and involutions}, Proc. Amer. Math. Soc.
	128 (2000), 1415-1418.
	
	\bibitem{BC} T. Bottazzi and C. Conde, \textit{A Grüss type operator inequality}. Ann. Funct. Anal. \textbf{8} (2017), no. 1, 124--132.
	
	\bibitem{BD} J. Bra\v{c}i\v{c} and C. Diogo, \textit{Relative numerical ranges}, Linear Algebra Appl. 485 (2015), 208--221.
	
	\bibitem{buzano} M. L. Buzano, \textit{Generalizzazione della diseguaglianza di Cauchy-Schwarz} (Italian), Rend. Sem. Mat. Univ. e Politech. Torino \textbf{31} (1974), 405-409.
	
	\bibitem{ChGLT} M-T. Chien, H-L. Gau, C-K. Li, M-C. Tsai, and K-Z. Wang,  Product of operators and numerical range. Linear Multilinear Algebra 64 (2016), no. 1, 58--67.
	
	\bibitem{Dra85} S. S. Dragomir, \textit{Some refinements of Schwartz inequality}, Simpozionul de Matematici \c si Aplica\c tii,
	Timi\c soara, Romania 1–2 (1985), 13–16.
	
	\bibitem{Dra07} S. S. Dragomir, \textit{Inequalities for the norm and the numerical radius of linear operators in Hilbert
		spaces}, Demonstratio Math. 40 (2) (2007), 411–417.
	
	\bibitem{Dra16}S.S.  Dragomir, \textit{Buzano  inequality holds for any projection} Bull. Aust. Math. Soc. 93 (2016), no. 3, 504--510. 
	
	
	\bibitem{Dra17} S.S.  Dragomir, \textit{A Buzano type inequality for two Hermitian forms and applications}, Linear Multilinear Algebra 65 (2017), no. 3, 514--525.
	
	\bibitem{DT} J. Duncan and P. J. Taylor, \textit{Norm inequalities for $C\sp*$-algebras}, Proc. Roy. Soc. Edinburgh Sect. A 75 (1975/76), no. 2, 119--129.
	
	\bibitem{fujiikubo} M. Fujii and F. Kubo, \textit{Buzano  inequality and bounds for roots of algebraic equations}, Proc. Amer. Math. Soc. \textbf{117} (1993), no. 2, 359--361.
	
	 \bibitem{Hal} P. R. Halmos, \textit{A Hilbert space problem book}, Van Nostrand, Princeton, 1967. 
	
	
	\bibitem{MKD} M. S.  Moslehian, M. Khosravi, and R. Drnov\u{s}ek, \textit{A commutator approach to Buzano  inequality}, Filomat {\bf26} (2012), no. 4, 827--832.
	\bibitem{PHD} K. Paul, S. Hossein and K. Das, \textit{Orthogonality on $B(H,H)$ and minimal-norm operator},  J. Anal. Appl. \textbf{6} (2008), no. 3, 169--178. 
	
	\bibitem{Pa} K. Paul, \textit{Translatable radii of an operator in the direction of another operator},  Sci. Math. \textbf{2} (1999), no. 1, 119--122.
	
	\bibitem{SMH} M. Sababheh, H. R. Moradi and Z. Heydarbeygi, \textit{Buzano, Krein and Cauchy-Schwarz inequalitites}, Oper. Matrices, \textbf{16}, (2022) 1,239--250.
	
	\bibitem{stampfli} J. G. Stampfli, \textit{The norm of a derivation}, Pacific J. Math. \textbf{33} (1970), 737--747.
	
\end{thebibliography}
\end{document}